\documentclass[a4paper,11pt]{amsart}

\usepackage{amsmath,amssymb,amsthm,latexsym,amscd,mathrsfs}
\usepackage{indentfirst}
\usepackage{stmaryrd}
\usepackage{graphicx}
\usepackage{extarrows}
\usepackage{multirow}
\usepackage{latexsym}
\usepackage{amsfonts}
\usepackage{color}
\usepackage{pictexwd,dcpic}
\usepackage{graphicx}
\usepackage{psfrag}
\usepackage{hyperref}
\usepackage{comment}
\usepackage{enumerate}

\numberwithin{equation}{section} \theoremstyle{plain}
\newtheorem{thm}{Theorem}[section]

\newtheorem{lem}[thm]{Lemma}
\newtheorem{cor}[thm]{Corollary}

\newtheorem*{acknow}{Acknowledgments}

\makeatletter

\def\<{\langle}
\def\>{\rangle}
\def\({\left(}
\def\){\right)}
\def\[{\left[}
\def\]{\right]}

\makeatother
\title[Compact Submanifolds with Pinched Ricci Curvature]
{Rigidity Results for Compact Submanifolds with Pinched Ricci Curvature in Euclidean and Spherical Space Forms}
\author[J.Q. Ge]{Jianquan Ge}
\address{School of Mathematical Sciences, Laboratory of Mathematics and Complex Systems, Beijing Normal
University, Beijing 100875, P.R. CHINA.}
\email{jqge@bnu.edu.cn}

\author[Y. Tao]{Ya Tao}
\address{Chern Institute of Mathematics and LPMC, Nankai University, Tianjin 300071, P. R. China.}
\email{tao-ya@mail.nankai.edu.cn}

\author[Y. Zhou]{Yi Zhou$^{*}$}
\address{Beijing International Center for Mathematical Research, Peking University,
Beijing 100871, P.R. CHINA.}
\email{yizhou@bicmr.pku.edu.cn}

\date{}
\thanks {$^{*}$ the corresponding author.}
\thanks{J. Q. Ge is partially supported by NSFC (No. 12171037, 12571049) and the Fundamental Research Funds for the Central Universities.}
\thanks{Y. Zhou is partially supported by  NSFC (No. 12171037, 12271040), China Postdoctoral Science Foundation (No. BX20230018)
and National Key R$\&$D Program of China 2020YFA0712800.}

\begin{document}
\maketitle

%%%%%%%%%%%%%%%%%%%%%
\begin{abstract}
For compact submanifolds in Euclidean and Spherical space forms with Ricci curvature bounded below by a function $\alpha(n,k,H,c)$ of mean curvature, we prove that the submanifold is either isometric to the Einstein Clifford torus, or a topological sphere for the maximal bound $\alpha(n,[\frac{n}{2}],H,c)$, or has up to $k$-th homology groups vanishing. This gives an almost complete (except for the differentiable sphere theorem) characterization of compact submanifolds with pinched Ricci curvature, generalizing celebrated rigidity results obtained by Ejiri, Xu-Tian, Xu-Gu, Xu-Leng-Gu, Vlachos, Dajczer-Vlachos.
\end{abstract}
%%%%%%%%%%%%%%%%%%%%%%%

%%%%%%%%%%%%%%%%%%%%%%
\section{Introduction}\label{sec1}
In 1979, Ejiri \cite{Eji79} obtained a celebrated rigidity theorem for $n(\geq4)$-dimensional, simply connected, compact, orientable, minimal immersed submanifolds in the unit sphere
with pinched Ricci curvature $\operatorname{Ric}\geq n-2$.
In 2011, Xu-Tian \cite{XT11} removed the assumption that the submanifold is simply connected.
Later, Xu-Gu \cite{XG13} generalized Ejiri's rigidity theorem to $n(\geq3)$-dimensional, compact, orientable submanifolds in a real space form $F^{n+m}(c)$ with $c+H^2>0$ under the assumption $\operatorname{Ric}\geq (n-2)\(c+H^2\)$ and parallel mean curvature vector field $\mathcal{H}$ ($H:=|\mathcal{H}|$). In the same paper, Xu-Gu raised a conjecture, indicating that the above generalization should also hold without the assumption of parallel mean curvature vector field.
Recently, Dajczer-Vlachos \cite{DV2} affirmed Xu-Gu's conjecture for the case $n\geq 4$,
except for the part of diffeomorphic to spheres.
The same philosophy was used in a series of rigidity results \cite{DJV, DV1, DV3} for submanifolds with pinched Ricci curvature in different settings.
For the case of spherical space forms \cite{DJV, DV1},
the core techniques are different improvements of the following theorem:

\begin{thm}\label{Thm}$($\cite[Theorem 2]{V07}$)$
Let $M^n$ be an $n(\geq4)$-dimensional compact submanifold of the unit sphere $\mathbf{S}^{n+m}$.
Assume that the Ricci curvature satisfies
\begin{equation}
\operatorname{Ric}>b(n,k,H):=\frac{n(k-1)}{k}+\frac{n(k-1)H}{2k^2}\(nH+\sqrt{n^2H^2+4k(n-k)}\),
\end{equation}
where $k$ is an integer such that $2\leq k\leq\[\frac{n}{2}\]$.
Then $\pi_1(M)=0$ and $H_p(M;\mathbb{Z})=H_{n-p}(M;\mathbb{Z})=0$ for all $1\leq p\leq k$.
\end{thm}

According to \cite{V07}, the function $b(n,k,H)$ comes from an expression for the Ricci curvature of
the standard immersion of the torus $S^k(r)\times S^{n-k}(\sqrt{1-r^2})$ into the unit sphere $\mathbf{S}^{n+1}$, where $S^k\(r\)$ denotes the $k$-dimensional sphere of radius $r$.
It is natural to compare $b(n,k,H)$ with another lower bound $\alpha(n,k,H,c)$ which was used in a similar homology vanishing result \cite[Lemma 4.1]{XLG14} under Ricci curvature pinching condition.
We then define and describe the function $\alpha(n,k,H,c)$.

Throughout this paper, we use the notation $F^n(c)$ to denote the simply connected, $n$-dimensional
real space form with constant sectional curvature $c$.
In addition, we set
$$\alpha(n,k,H,c):=\bigg(n-1-\dfrac{k(n-k)(n-2)}{k(n-2k)+n(k-1)(n-k)}\bigg)\(c+H^2\).$$
For the Einstein torus $S^k\(\sqrt{\frac{k-1}{n-2}}\,r\)\times S^{n-k}\(\sqrt{\frac{n-k-1}{n-2}}\,r\)\subset S^{n+1}\(r\)$  $(2\leq k\leq\[\frac{n}{2}\])$
embedded in $F^{n+m}(c)$ with mean curvature $H$,
a direct calculation shows that its Ricci curvature can be expressed by $\operatorname{Ric}=\frac{n-2}{r^2}=\alpha(n,k,H,c)$.

Now we compare $\alpha(n,k,H,1)$ and $b(n,k,H)$ for $n\geq4$ and $2\leq k\leq\[\frac{n}{2}\]$.
By a direct calculation, we obtain
$$\begin{aligned}
b(n,k,H)-\alpha(n,k,H,1)=&\frac{n(k-1)}{2k^2}H\sqrt{n^2H^2+4k(n-k)}\\
&+\frac{n(k-1)(n-2k)}{k\big((n+2)\gamma(k)+2n\big)}\bigg(\frac{n}{2k}\gamma(k)H^2-(n-k)\bigg),
\end{aligned}$$
where $\gamma(k):=k(n-k)-n\geq n-4\geq0$. Hence we have the following results.
\begin{itemize}
\item[$(1)$] $\alpha(n,k,0,1)\geq b(n,k,0)$ and the equality holds if and only if $n=2k$.
\item[$(2)$] $\alpha(n,k,H,1)< b(n,k,H)$ for sufficiently large $H$.
\item[$(3)$] If $n=2k$, then we have $\alpha(n,\frac{n}2,H,1)\leq b(n,\frac{n}2,H)$
and the equality holds if and only if $H=0$.
\end{itemize}

In the following, we state the main theorem of this paper
which improves the pinched condition of \cite[Lemma 4.1]{XLG14} and presents a complete rigidity result.
\begin{thm}\label{ThmA}
Let $f: M^n\rightarrow F^{n+m}(c)$, $n\geq5$, $c\geq0$, be an isometric immersion of a compact, connected Riemannian manifold.
If the Ricci curvature of $M$ satisfies
\begin{equation*}\label{Pinchcond}
\operatorname{Ric}\geq\alpha(n,k,H,c)
=\bigg(n-1-\dfrac{k(n-k)(n-2)}{k(n-2k)+n(k-1)(n-k)}\bigg)\(c+H^2\)\tag{$*$}
\end{equation*}
for an integer $k$ with $2\leq k\leq\[\frac{n}{2}\]$, then one of the following holds:
\begin{itemize}
\item[$(1)$] $\pi_1(M)=0$ and
$$H_q(M;\mathbb{Z})=H_{n-q}(M;\mathbb{Z})=0\ \mbox{for each}\ 1\leq q\leq k.$$
If in addition $k<\[\frac{n}{2}\]$, then $H_{n-k-1}(M;\mathbb{Z})$ is torsion-free.
\item[$(2)$] $M$ is isometric to the Einstein hypersurface
$S^k\(\sqrt{\frac{k-1}{n-2}}\,r\)\times S^{n-k}\(\sqrt{\frac{n-k-1}{n-2}}\,r\)$
in the totally umbilical sphere $S^{n+1}\(r\)\subset F^{n+m}(c)$, where $r=\sqrt{\frac{n-2}{\alpha(n,k,H,c)}}$ with constant $H$.
In this case, $\operatorname{Ric}=\alpha(n,k,H,c)$ on $M$.
\end{itemize}
\end{thm}

The next corollary analyses the special case $k=\[\frac{n}{2}\]$ in Theorem \ref{ThmA}.
Thus we can give another proof for \cite[Theorem 1]{DV2} when $n\geq5$.
Different from Dajczer-Vlachos' method, when $M$ is not a topological sphere,
we show that $M$ has parallel mean curvature vector field by a careful study of the first normal bundle.

\begin{cor}\label{CorA}
Let $f: M^n\rightarrow F^{n+m}(c)$, $n\geq5$, $c\geq0$, be an isometric immersion of a compact, connected Riemannian manifold. If the Ricci curvature of $M$ satisfies
$$\operatorname{Ric}\geq(n-2)\(c+H^2\),$$
then $M$ is either homeomorphic to the $n$-sphere or isometric to the minimal Clifford hypersurface
$S^{k}\(r/\sqrt{2}\)\times S^{k}\(r/\sqrt{2}\)$
in the totally umbilical sphere $S^{n+1}\(r\)\subset F^{n+m}(c)$,
where $n=2k$ and $r=\frac{1}{\sqrt{c+H^2}}$.
\end{cor}

For the case of odd-dimensional manifold, we can obtain the following rigidity result under a weaker pinched condition, which improves \cite[Theorem 1.4]{XLG14}.
Other rigidity results for odd-dimensional submanifolds in the unit sphere are also obtained in \cite{DJV, DV1, V02}.

\begin{cor}\label{CorB}
Let $f: M^n\rightarrow F^{n+m}(c)$, $n\geq5$, $c\geq0$, be an isometric immersion of a compact, connected, odd-dimensional Riemannian manifold. If the Ricci curvature of $M$ satisfies
$$\operatorname{Ric}\geq\(n-2-\frac{4}{n^3-2n^2-n-2}\)\(c+H^2\),$$
then $M$ is either homeomorphic to the $n$-sphere or isometric to the Einstein hypersurface
$S^k\(\sqrt{\frac{k-1}{n-2}}\,r\)\times S^{n-k}\(\sqrt{\frac{n-k-1}{n-2}}\,r\)$
in the totally umbilical sphere $S^{n+1}\(r\)\subset F^{n+m}(c)$,
where $n=2k+1$ and $r=\sqrt{\frac{n-2}{\alpha(n,k,H,c)}}$.
\end{cor}

The rest of this paper is organized as follows.
In Section \ref{sec2}, we give some preliminaries and show the core lemma of this paper
(Lemma \ref{Ricci pinching}).
In Section \ref{sec3}, we prove Theorem \ref{ThmA}, Corollary \ref{CorA} and Corollary \ref{CorB}.

\section{Preliminaries}\label{sec2}
Let $f:M^n\rightarrow F^{n+m}(c)$ be an isometric immersion,
and let $\operatorname{II}: TM\times TM\rightarrow NM$ be the second fundamental form.
For an arbitrary fixed point $x\in M$,
let $\{e_1,\cdots,e_n\}$ and $\{\xi_1,\cdots,\xi_m\}$ be orthonormal bases of
$T_xM$ and $N_xM$, respectively.
Write $$\operatorname{II}(e_i, e_j)=\sum_{\alpha=1}^mh_{ij}^\alpha\xi_\alpha\
\mbox{for}\ 1\leq i, j\leq n.$$
For each $1\leq\alpha \leq m$, let $$H_\alpha:=\(h_{ij}^\alpha\)_{n\times n}\ \mbox{and}\ c_\alpha:=\frac1n\operatorname{tr}H_\alpha=\frac1n\sum_{i=1}^nh_{ii}^\alpha.$$
Then the mean curvature vector field is given by
\begin{equation}\label{Hfield}
\mathcal{H}=\frac1n\sum_{i=1}^n\operatorname{II}(e_i, e_i)
=\frac1n\sum_{\alpha=1}^m\sum_{i=1}^nh_{ii}^\alpha\xi_\alpha
=\sum_{\alpha=1}^mc_\alpha\xi_\alpha.
\end{equation}
Let $H=|\mathcal{H}|$ be the mean curvature, and let
$S=|\operatorname{II}|^2$ be the squared length of the second fundamental form. %=|\mathcal{H}|
Then the Gauss equation implies that
\begin{equation}\label{RicGauss}
\operatorname{Ric}(e_i, e_j)=(n-1)c\delta_{ij}
+\sum_{\alpha=1}^m\bigg(nc_\alpha h_{ij}^\alpha-\sum_{k=1}^nh_{ik}^\alpha h_{jk}^\alpha\bigg),
\end{equation}
\begin{equation}\label{scalarGauss}
\rho=n(n-1)c+n^2H^2-S,
\end{equation}
where $\operatorname{Ric}$ and $\rho$ are the Ricci tensor and the scalar curvature of $M$, respectively.
In addition, we use the notation $\operatorname{Ric}(X):=\operatorname{Ric}(X, X)$
for $X\in TM$ with $|X|=1$.
For each $x\in M$, the subspace $$N_1(x)=\operatorname{Span}\{\operatorname{II}(X,Y): X,Y\in T_xM\}$$
is called the first normal space of $f$ at $x$.
An isometric immersion $f$ is said to be $1$-regular if
the first normal space $N_1(x)$ has constant dimension independent of $x\in M$.
For a $1$-regular isometric immersion $f$, the first normal spaces form a subbundle $N_1$ of the normal bundle $NM$, called the first normal bundle of $f$.

Besides, we also need the following terminology.
A normal vector $\eta$ at $x\in M$ is called a principal normal of $f$ at $x$ if the subspace
$$E_\eta(x):=\left\{X\in T_xM: \operatorname{II}(X,Y)=\<X,Y\>\eta\ \mbox{for each}\ Y\in T_xM \right\}$$
is nontrivial. A smooth normal vector field $\eta$ on $M$ is called a principal normal vector field of $f$ with multiplicity $k>0$ if $E_\eta(x)$ has constant dimension $k$
(and hence is a smooth distribution).
If in addition $\eta$ is parallel (with respect to the normal connection) along $E_\eta$,
then it is called a Dupin principal normal vector field.
According to \cite[Proposition 1.22]{DT19},
every principal normal vector field with multiplicity $k\geq 2$ is a Dupin principal normal vector field.

Next, we will present several lemmas used in the proof of Theorem \ref{ThmA} and Corollary \ref{CorA}.
The following Lemma \ref{homology vanishing} was first proved by Lawson-Simons \cite{LS73}
for the case $c>0$, and then by Xin \cite{Xin84} for the case $c=0$
when the inequality (\ref{Vanishcond}) is strict on the whole submanifold.
Elworthy-Rosenberg \cite{ER96} observed that the result still holds by only requiring
the inequality (\ref{Vanishcond}) is strict at one point.

\begin{lem}\label{homology vanishing}$($\cite{ER96, LS73, Xin84}$)$. %, $n\geq4$,
Let $f: M^n\rightarrow F^{n+m}(c)$, $c\geq0$, be an isometric immersion of a compact manifold.
Suppose that
\begin{equation*}\label{Vanishcond}
\Theta_q:=\sum_{i=1}^q\sum_{j=q+1}^n\(2|\operatorname{II}(e_i,e_j)|^2
-\<\operatorname{II}(e_i,e_i),\operatorname{II}(e_j,e_j)\>\)\leq q(n-q)c\tag{$\#$}
\end{equation*}
holds for any point $x\in M$ and any orthonormal basis $\{e_1,\cdots,e_n\}$ of $T_xM$,
where $q$ is an integer such that $1\leq q\leq n-1$.
If there exists a point of $M$ such that the inequality $($\ref{Vanishcond}$)$ is strict, then
$H_q(M;\mathbb{Z})=H_{n-q}(M;\mathbb{Z})=0$.
\end{lem}

Using Lemma \ref{homology vanishing},
Dajczer-Vlachos \cite{DV2} improved \cite[Lemma 2.2]{XG13} to the following Lemma \ref{first vanishing 1}.
\begin{lem}$($\cite[Lemma 4]{DV2}$)$\label{first vanishing 1}
Let $f: M^n\rightarrow F^{n+m}(c)$, $n\geq4$, $c\geq0$, be an isometric immersion of a compact manifold.
Suppose that the Ricci curvature of $M$ satisfies
\begin{equation}\label{LemXG}
\operatorname{Ric}\geq\frac{n(n-1)}{n+2}\(c+H^2\).
\end{equation}
If there exists a point of $M$ such that the inequality $(\ref{LemXG})$ is strict,
then $\pi_1(M)=0$ and $H_1(M;\mathbb{Z})=H_{n-1}(M;\mathbb{Z})=0$.
\end{lem}

The following Lemma \ref{alpha} concerns the estimate of $\alpha(n,k,H,c)$,
which will be used several times later.
For $n\geq5$, let $\varphi: [2,\frac{n}{2}]\rightarrow\mathbb{R}$ be the smooth function given by
$$\varphi(s):=\frac{s(n-s)}{s(n-2s)+n(s-1)(n-s)}.$$
Then we have $$\alpha(n,k,H,c)=\big(n-1-(n-2)\varphi(k)\big)\(c+H^2\),$$
and we will only consider the case of $c+H^2>0$ in the following lemma.
\begin{lem}\label{alpha}
The function $\varphi(s)$ is strictly monotone decreasing, and thus
$\alpha(n,k,H,c)$ is strictly monotone increasing for $2\leq k\leq\[\frac{n}{2}\]$.
In addition, we have
\begin{equation}\label{estalpha}
(n-3)\(c+H^2\)<\alpha(n,k,H,c)\leq(n-2)\(c+H^2\)
\end{equation}
for any $2\leq k\leq\[\frac{n}{2}\]$, and the equality holds if and only if $n=2k$.
\end{lem}
\begin{proof}
A direct calculation shows that $$\varphi'(s)=\frac{-n^2(n-2s)}{(s(n-2s)+n(s-1)(n-s))^2}<0$$
for $s\in[2,\frac{n}{2})$.
It follows that $\varphi(s)$ is strictly monotone decreasing, and thus
$\alpha(n,k,H,c)$ is strictly monotone increasing for $2\leq k\leq\[\frac{n}{2}\]$.
In addition, we have
\begin{equation}\label{estalpha1}
\alpha(n,k,H,c)\leq\big(n-1-(n-2)\varphi(n/2)\big)\(c+H^2\)=(n-2)\(c+H^2\),
\end{equation}
\begin{equation}\label{estalpha2}
\begin{aligned}
\alpha(n,k,H,c)&\geq\big(n-1-(n-2)\varphi(2)\big)\(c+H^2\)=\big(n-1-\frac{2(n-2)^2}{n^2-8}\big)\(c+H^2\)\\
&>(n-3)\(c+H^2\),
\end{aligned}
\end{equation}
where the last inequality in (\ref{estalpha2}) follows from the fact that $(n-2)^2<n^2-8$ for $n\geq5$.
It is easy to see that the equality holds in (\ref{estalpha1}) if and only if $n=2k$.
\end{proof}

In the study of the first normal bundle $N_1$,
we need the following Lemma \ref{N1para} to show that $N_1$ is parallel.
\begin{lem}\label{N1para}$($\cite[Lemma 5]{RT84}$)$
Let $f:M^n\rightarrow F^{n+m}(c)$ be a $1$-regular isometric immersion. If for every point of $M$, there exists an orthonormal frame $\{E_1,\cdots, E_n\}$ such that for each $1\leq i\leq n$,
$\operatorname{II}(E_i,E_j)=0$ for any $j\neq i$ and
$\operatorname{II}(E_i,E_i)=\sum_{j\neq i}a_j\operatorname{II}(E_j,E_j)$ for some real numbers $a_j$'s,
then the first normal bundle $N_1$ is parallel.
\end{lem}

Finally, we state and prove the core lemma of this paper
which is based on the proof of \cite[Lemma 4.1]{XLG14}.
\begin{lem}\label{Ricci pinching}
Let $f:M^n\rightarrow F^{n+m}(c)$, $n\geq 5$, be an isometric immersion satisfying \begin{equation*}\label{Pinchcond}
\operatorname{Ric}\geq\alpha(n,k,H,c)
=\bigg(n-1-\dfrac{k(n-k)(n-2)}{k(n-2k)+n(k-1)(n-k)}\bigg)\(c+H^2\)\tag{$*$}
\end{equation*}
at a point $x\in M$ for an integer $k$ such that $2\leq k\leq\[\frac{n}{2}\]$.
\begin{itemize}
\item[(1)] The inequality $($\ref{Vanishcond}$)$ is satisfied for $2\leq q\leq k$ at $x$
and is strict for $2\leq q<k$ if $c+H^2>0$ at $x$.
Moreover, if the inequality $($\ref{Pinchcond}$)$ is strict at $x$,
then the inequality $($\ref{Vanishcond}$)$ is strict for $q=k$ at $x$.
\item[(2)] Suppose that the equality holds in $($\ref{Vanishcond}$)$ with $q=k$
for an orthonormal basis $\{e_1,\cdots,e_n\}$ of $T_xM$.
Then for any orthonormal basis $\{\xi_1,\cdots,\xi_m\}$ of $N_xM$, the corresponding matrix
$$H_\alpha=\operatorname{diag}(\lambda_\alpha I_k,\mu_\alpha I_{n-k})\
\mbox{for each}\ 1\leq\alpha\leq m,$$
where $\lambda_\alpha$ and $\mu_\alpha$ are principal curvatures with respect to $\xi_\alpha$
of multiplicities $k$ and $n-k$, respectively.
Furthermore, if $c+H^2>0$ at $x$, then
$\eta_1=\sum_{\alpha=1}^m\lambda_\alpha\xi_\alpha$ and $\eta_2=\sum_{\alpha=1}^m\mu_\alpha\xi_\alpha$
are distinct principal normals of $f$ at $x$ such that $E_{\eta_1}(x)=\operatorname{Span}\{e_1,\cdots,e_k\}$ and $E_{\eta_2}(x)=\operatorname{Span}\{e_{k+1},\cdots,e_n\}$.
In addition, $$\operatorname{Ric}(e_i,e_j)=\delta_{ij}\alpha(n,k,H,c)\ \mbox{for any}\ 1\leq i,j\leq n.$$
\end{itemize}
\end{lem}
\begin{proof}
To prove (1), we will estimate $\Theta_q$ for any $2\leq q\leq k$.
Let $\{e_1,\cdots,e_n\}$ and $\{\xi_1,\cdots,\xi_m\}$ be orthonormal bases of
$T_xM$ and $N_xM$, respectively.
It follows from (\ref{Hfield}-\ref{RicGauss}) that $\sum\limits_{\alpha=1}^m c_\alpha^2=H^2$ and
\begin{equation}\label{RicGauss1}
\operatorname{Ric}(e_i)=(n-1)c+\sum_{\alpha=1}^m\bigg(nc_\alpha h^\alpha_{ii}
-\sum_{j=1}^n(h^\alpha_{ij})^2\bigg).
\end{equation}
On one hand, we have
\begin{equation}\label{scaling1}
\begin{aligned}
\Theta_q&=2\sum_{\alpha=1}^m\sum_{j=q+1}^n\sum_{i=1}^q(h_{ij}^\alpha)^2
-\sum_{\alpha=1}^m\sum_{j=q+1}^n\sum_{i=1}^qh_{ii}^\alpha h_{jj}^\alpha\\
&=\sum_{\alpha=1}^m\bigg(2\sum_{j=q+1}^n\sum_{i=1}^q(h_{ij}^\alpha)^2-\big(\sum_{i=1}^qh_{ii}^\alpha\big)\big(nc_\alpha-\sum_{i=1}^qh_{ii}^\alpha\big)\bigg)\\
&\leq\sum_{\alpha=1}^m\bigg(2\sum_{j=q+1}^n\sum_{i=1}^q(h_{ij}^\alpha)^2-nc_\alpha\sum_{i=1}^qh_{ii}^\alpha+q\sum_{i=1}^q(h_{ii}^\alpha)^2\bigg)\\
&\leq q\sum_{i=1}^q\big((n-1)c-\operatorname{Ric}(e_i)\big)+n(q-1)\sum_{\alpha=1}^m\sum_{i=1}^qc_\alpha h_{ii}^\alpha\\
&\leq q^2\big((n-1)\(c+H^2\)-\operatorname{Ric}_{\min}(x)\big)\\
&\quad\,-q(n-q)H^2+n(q-1)\sum_{\alpha=1}^m\sum_{i=1}^qc_\alpha(h_{ii}^\alpha-c_\alpha),
\end{aligned}
\end{equation}
where $\operatorname{Ric}_{\min}(x):=\min\left\{\operatorname{Ric}(X): X\in T_xM, |X|=1\right\}$.
On the other hand, we have

\begin{equation}\label{scaling2}
\begin{aligned}
\Theta_q
&=\sum_{\alpha=1}^m\bigg(2\sum_{j=q+1}^n\sum_{i=1}^q(h_{ij}^\alpha)^2-\frac{n-q}{n}\big(\sum_{i=1}^qh_{ii}^\alpha\big)\big(nc_\alpha-\sum_{i=1}^qh_{ii}^\alpha\big)\\
&\quad\quad\quad\quad\quad\quad\quad\quad\quad\quad\ \ -\frac{q}{n}\big(\sum_{j=q+1}^nh_{jj}^\alpha\big)\big(nc_\alpha-\sum_{j=q+1}^nh_{jj}^\alpha\big)\bigg)\\
&\leq\sum_{\alpha=1}^m\bigg(2\sum_{j=q+1}^n\sum_{i=1}^q(h_{ij}^\alpha)^2-(n-q)c_\alpha\sum_{i=1}^qh_{ii}^\alpha+\frac{q(n-q)}{n}\sum_{i=1}^q(h_{ii}^\alpha)^2\\
&\quad\quad\quad\quad\quad\quad\quad\quad\quad\quad\ \ -qc_\alpha\sum_{j=q+1}^nh_{jj}^\alpha+\frac{q(n-q)}{n}\sum_{j=q+1}^n(h_{jj}^\alpha)^2\bigg)\\
&\leq\frac{q(n-q)}{n}S-\sum_{\alpha=1}^m\bigg(qnc_\alpha^2+(n-2q)c_\alpha\sum_{i=1}^qh_{ii}^\alpha\bigg)\\
&\leq q(n-q)\bigg((n-1)\(c+H^2\)-\operatorname{Ric}_{\min}(x)\bigg)\\
&\quad\,-q(n-q)H^2-(n-2q)\sum_{\alpha=1}^m\sum_{i=1}^qc_\alpha(h_{ii}^\alpha-c_\alpha).
\end{aligned}
\end{equation}
From (\ref{scaling1}-\ref{scaling2}), the assumption (\ref{Pinchcond}) and Lemma \ref{alpha}, we obtain
\begin{equation}\label{scaling}
\begin{aligned}
\Theta_q&\leq\frac{n-2q}{q(n-2)}\bigg(q^2\big((n-1)\(c+H^2\)-\operatorname{Ric}_{\min}(x)\big)\\
&\qquad\qquad\quad\ \ -q(n-q)H^2+n(q-1)\sum_{\alpha=1}^m\sum_{i=1}^qc_\alpha(h_{ii}^\alpha-c_\alpha)\bigg)\\
&\quad+\frac{n(q-1)}{q(n-2)}\bigg(q(n-q)\big((n-1)\(c+H^2\)-\operatorname{Ric}_{\min}(x)\big)\\
&\qquad\qquad\qquad\ \, -q(n-q)H^2-(n-2q)\sum_{\alpha=1}^m\sum_{i=1}^qc_\alpha(h_{ii}^\alpha-c_\alpha)\bigg)\\
&=\frac{q(n-2q)+n(q-1)(n-q)}{n-2}\big((n-1)\(c+H^2\)-\operatorname{Ric}_{\min}(x)\big)\\
&\qquad\qquad\qquad\qquad\qquad\qquad\qquad\qquad\qquad\qquad\quad\ \ -q(n-q)H^2\\
&\leq\frac{\big(q(n-2q)+n(q-1)(n-q)\big)k(n-k)}{k(n-2k)+n(k-1)(n-k)}\(c+H^2\)-q(n-q)H^2\\
&\leq\frac{\big(q(n-2q)+n(q-1)(n-q)\big)q(n-q)}{q(n-2q)+n(q-1)(n-q)}\(c+H^2\)-q(n-q)H^2\\
&=q(n-q)c
\end{aligned}
\end{equation}
for any $2\leq q\leq k$.
Note that the last inequality in (\ref{scaling}) is strict for $2\leq q<k$ if $c+H^2>0$ at $x$.
Hence the inequality (\ref{Vanishcond}) is satisfied for $2\leq q\leq k$ at $x$
and is strict for $2\leq q<k$ if $c+H^2>0$ at $x$.
If in addition the inequality (\ref{Pinchcond}) is strict at $x$, then the second inequality in (\ref{scaling})
is strict, and thus the inequality (\ref{Vanishcond}) is strict for $q=k$ at $x$.

Next, we prove $(2)$. Assume that the equality holds in (\ref{Vanishcond}) with $q=k$ for an orthonormal basis $\{e_1,\cdots,e_n\}$ of $T_xM$.
Then all inequalities in (\ref{scaling1}-\ref{scaling}) become equalities.
The first inequality in (\ref{scaling2}) implies that for each $1\leq\alpha\leq m$, there exist
$\lambda_\alpha, \mu_\alpha\in\mathbb{R}$ such that
\begin{equation}\label{eqcondition1}
h_{ii}^\alpha=\lambda_\alpha\ \mbox{for any}\ 1\leq i\leq k,
\end{equation}
\begin{equation}\label{eqcondition2}
h_{jj}^\alpha=\mu_\alpha\ \mbox{for any}\ k+1\leq j\leq n,
\end{equation}
and the second inequality in (\ref{scaling2}) implies that for each $1\leq\alpha\leq m$,
\begin{equation}\label{eqcondition3}
\bigg(\frac{k(n-k)}{n}-1\bigg)h_{ij}^\alpha=0\ \mbox{for any}\ 1\leq i\leq k,\ k+1\leq j\leq n,
\end{equation}
\begin{equation}\label{eqcondition4}
h_{ii'}^\alpha=h_{jj'}^\alpha=0\ \mbox{for any}\ 1\leq i\neq i'\leq k,\ k+1\leq j\neq j'\leq n.
\end{equation}
Since $2\leq k\leq\[\frac{n}{2}\]$, we have $\frac{k(n-k)}{n}>1$ for $n\geq5$,
and thus it follows from (\ref{eqcondition1}-\ref{eqcondition4}) that
$$H_\alpha=\operatorname{diag}(\lambda_\alpha I_k,\mu_\alpha I_{n-k})\
\mbox{for each}\ 1\leq\alpha\leq m,$$
Then by (\ref{RicGauss}), we have
$$\operatorname{Ric}(e_i,e_j)=0\ \mbox{for any}\ 1\leq i\neq j\leq n.$$
In addition, combining the last inequality in (\ref{scaling2}) and the second inequality in (\ref{scaling}), we obtain
\begin{equation}\label{Ric(ei)}
\operatorname{Ric}(e_i)=\alpha(n,k,H,c)\ \mbox{for each}\ 1\leq i\leq n.
\end{equation}
Furthermore, we consider the normal vectors
$\eta_1=\sum_{\alpha=1}^m\lambda_\alpha\xi_\alpha$ and $\eta_2=\sum_{\alpha=1}^m\mu_\alpha\xi_\alpha$.
If $\eta_1=\eta_2$, then by (\ref{RicGauss}), we have
$\operatorname{Ric}(e_i)=(n-1)\(c+H^2\)$ for each $1\leq i\leq n$,
which contradicts (\ref{Ric(ei)}) and Lemma \ref{alpha} if $c+H^2>0$ at $x$.
Hence, $\eta_1\neq\eta_2$ and
for any tangent vectors $X=\sum_{i=1}^nx_ie_i, Y=\sum_{i=1}^ny_ie_i\in T_xM$, we have
$$\begin{aligned}
\operatorname{II}(X, Y)&=\sum_{\alpha=1}^m\sum_{i,j=1}^nx_iy_jh_{ij}^\alpha\xi^\alpha
=\sum_{\alpha=1}^m\Big(\sum_{i=1}^kx_iy_i\lambda_\alpha\xi^\alpha
+\sum_{j=k+1}^nx_jy_j\mu_\alpha\xi^\alpha\Big)\\
&=\Big(\sum_{i=1}^kx_iy_i\Big)\eta_1+\Big(\sum_{j=k+1}^nx_jy_j\Big)\eta_2.
\end{aligned}$$
It follows that $\eta_1$, $\eta_2$ are principal normals of $f$ at $x$ such that
$E_{\eta_1}(x)=\operatorname{Span}\{e_1,\cdots,e_k\}$ and $E_{\eta_2}(x)=\operatorname{Span}\{e_{k+1},\cdots,e_n\}$.
\end{proof}

\section{Proof of Theorems}\label{sec3}
In this section, we will prove Theorem \ref{ThmA}, Corollary \ref{CorA} and Corollary \ref{CorB}.

\subsection{Proof of Theorem \ref{ThmA}}
\begin{proof}
First, we will prove $\pi_1(M)=0$ and $H_1(M;\mathbb{Z})=H_{n-1}(M;\mathbb{Z})=0$
by Lemma \ref{first vanishing 1}.
The first line of $(\ref{estalpha2})$ yields that
\begin{equation}\label{estim1}
\alpha(n,k,H,c)\geq\bigg(n-1-\frac{2(n-2)^2}{n^2-8}\bigg)\(c+H^2\)\geq\frac{n(n-1)}{n+2}\(c+H^2\),
\end{equation}
where the last inequality follows from the fact that $n-1-\frac{2(n-2)^2}{n^2-8}>\frac{n(n-1)}{n+2}$
for $n\geq5$.
%Note that if $n\geq 4$ then the coefficient equality holds iff $n=4$.
Then by the assumption $\operatorname{Ric}\geq\alpha(n,k,H,c)$,
we only need to show that there exists a point of $M$ such that the inequality $(\ref{LemXG})$ is strict.
Assume for the sake of contradiction that
$\operatorname{Ric}=\frac{n(n-1)}{n+2}\(c+H^2\)$ on $M$.
Then the last inequality in (\ref{estim1}) implies that $c=H=0$,
which contradicts the compactness of $M$.

Next, it follows immediately from Lemma \ref{Ricci pinching} (1) that
the inequality (\ref{Vanishcond}) is satisfied for $2\leq q\leq k$ at every point of $M$ and strict for $2\leq p<k$ at points with $c+H^2>0$.
Then Lemma \ref{homology vanishing} implies that
$$H_q(M;\mathbb{Z})=H_{n-q}(M;\mathbb{Z})=0\ \mbox{for each}\ 2\leq q\leq k-1.$$
For $H_k(M;\mathbb{Z})$ and $H_{n-k}(M;\mathbb{Z})$,
we divide the discussion into the following two cases.

Case (1): $H_k(M;\mathbb{Z})=H_{n-k}(M;\mathbb{Z})=0$.
In this case, if $k=\[\frac{n}2\]$, then we know that $M$ is a homology sphere.
And if $k<\[\frac{n}2\]$, then by the universal coefficient theorem,
$H_{k}(M;\mathbb{Z})=0$ implies that $H^{k+1}(M;\mathbb{Z})$ is torsion-free,
and thus $H_{n-k-1}(M;\mathbb{Z})$ is also torsion-free by the Poincar\'{e} duality theorem.

Case (2): $H_k(M;\mathbb{Z})=H_{n-k}(M;\mathbb{Z})=0$ does not hold.
It follows from Lemma \ref{homology vanishing} that for any point $x\in M$,
there exists an orthonormal basis $\{e_1,\cdots,e_n\}$ of $T_xM$ such that %$\{e_1(x),\cdots,e_n(x)\}$
the equality holds in (\ref{Vanishcond}) for $q=k$.
Then by Lemma \ref{Ricci pinching} (2), for any smooth orthonormal local frame $\(\xi_1,\cdots,\xi_m\)$ for $NM$ over an open subset $U$ of $M$, %and any point $x\in U$,
the corresponding matrix
\begin{equation}\label{shapeop}
H_\alpha=\operatorname{diag}(\lambda_\alpha I_k,\mu_\alpha I_{n-k})\
\mbox{for each}\ 1\leq\alpha\leq m,
\end{equation}
where $\lambda_\alpha$ and $\mu_\alpha$ are principal curvatures with respect to $\xi_\alpha$
of multiplicities $k$ and $n-k$, respectively.
It follows that the normal bundle $NM$ is flat.
In addition, by Lemma \ref{Ricci pinching} (2), we have
\begin{equation}\label{Einstein}
\operatorname{Ric}(e_i,e_j)=\delta_{ij}\alpha(n,k,H,c)\ \mbox{for any}\ 1\leq i,j\leq n.
\end{equation}
Since $M$ is connected and $n\geq5$, by Schur's lemma,
we obtain that $M$ is an Einstein manifold and $\alpha(n,k,H,c)$ is constant.
It follows that the mean curvature $H$ is also constant, and thus $c+H^2>0$ at every point of $M$.
Note that the expressions of the principal normals
$$\eta_1=\sum_{\alpha=1}^m\lambda_\alpha\xi_\alpha\ \mbox{and}\ \eta_2=\sum_{\alpha=1}^m\mu_\alpha\xi_\alpha$$
are independent of the choice of the smooth orthonormal local frame $\(\xi_1,\cdots,\xi_m\)$ for $NM$.
Hence, there are two global normal vector fields, still denoted by $\eta_1$ and $\eta_2$,
which are distinct principal normals of $f$ at every point.

In order to obtain the rigidity result by \cite[Theorem 1]{Onti18},
we need to prove that the mean curvature vector field $\mathcal{H}$ is parallel.
For the case $H>0$, we can assume without loss of generality that $\xi_1=\mathcal{H}/{H}$.
Then we have $\operatorname{tr}H_1=nH$ and $\operatorname{tr}H_\alpha=0$ for each $2\leq\alpha\leq m$.
It follows from (\ref{shapeop}) that
\begin{equation}\label{eqtrH1}
k\lambda_1+(n-k)\mu_1=nH,
\end{equation}
\begin{equation}\label{eqtrH2}
k\lambda_\alpha+(n-k)\mu_\alpha=0
\end{equation}
for each $2\leq\alpha\leq m$.
By (\ref{RicGauss}) and (\ref{Einstein}), we have
\begin{equation}\label{RicGauss1}
\alpha(n,k,H,c)=(n-1)c+nH\lambda_1-\lambda_1^2-\sum_{\alpha=2}^m\lambda_\alpha^2,
\end{equation}
\begin{equation}\label{RicGauss2}
\alpha(n,k,H,c)=(n-1)c+nH\mu_1-\mu_1^2-\sum_{\alpha=2}^m\mu_\alpha^2.
\end{equation}
Combining (\ref{eqtrH1}-\ref{RicGauss2}), one can verify by a direct calculation that
\begin{equation}\label{Prin1}
\lambda_1=H+(n-2k)\frac{\varphi(k)\(c+H^2\)}{kH},
\end{equation}
\begin{equation}\label{Prin2}
\mu_1=H-(n-2k)\frac{\varphi(k)\(c+H^2\)}{(n-k)H}
\end{equation}
are constant. It follows that
$\sum\limits_{\alpha=2}^m\lambda_\alpha^2, \sum\limits_{\alpha=2}^m\mu_\alpha^2$ are constant,
and thus
$$|\eta_1|^2=\lambda_1^2+\sum_{\alpha=2}^m\lambda_{\alpha}^2,\ |\eta_2|^2=\mu_1^2+\sum_{\alpha=2}^m\mu_{\alpha}^2,$$
$$\<\eta_1,\eta_2\>=\lambda_1\mu_1+\sum_{\alpha=2}^m\lambda_\alpha\mu_\alpha
=\lambda_1\mu_1-\frac{k}{n-k}\sum_{\alpha=2}^m\lambda_\alpha^2$$
are also constant.
Since $N_1(x)=\operatorname{Span}\{\eta_1(x),\eta_2(x)\}$ for each $x\in M$
and the determinant of the Gram matrix $\(\<\eta_i,\eta_j\>\)_{2\times2}$ is constant,
we know that $f$ is $1$-regular and $\operatorname{rank}N_1=1$ or $2$. %$1\leq\operatorname{rank}N_1\leq2$.
To prove the mean curvature vector field $\mathcal{H}$ is parallel,
we divide the discussion into the following two cases.

Case (2-a): $\operatorname{rank}N_1=1$.
In this case, it follows from (\ref{eqtrH1}-\ref{RicGauss2}) that
$\lambda_1\mu_1\neq0$ as $H>0$ and $\lambda_\alpha=\mu_\alpha=0$ for each $2\leq\alpha\leq m$.
Then $\eta_1=\lambda_1\xi_1$ and $\eta_2=\mu_1\xi_1$ are smooth normal vector fields
and hence Dupin principal normal vector fields with multiplicities $k$ and $n-k$, respectively.
Therefore, $E_{\eta_1}$ and $E_{\eta_2}$ are smooth distributions such that $TM=E_{\eta_1}\oplus E_{\eta_2}$.
For every point of $M$, we can choose a smooth orthonormal local frame $\(E_1,\cdots,E_n\)$ such that $E_{\eta_1}=\operatorname{Span}\{E_1,\cdots,E_k\}$ and $E_{\eta_2}=\operatorname{Span}\{E_{k+1},\cdots,E_n\}$.
It follows that
$$\nabla_{E_i}^\bot\xi_1=\frac1{\lambda_1}\nabla_{E_i}^\bot\eta_1=0\
\mbox{for each}\ 1\leq i\leq k,$$
$$\nabla_{E_j}^\bot\xi_1=\frac1{\mu_1}\nabla_{E_j}^\bot\eta_2=0\
\mbox{for each}\ k+1\leq j\leq n,$$
and thus $\mathcal{H}=H\xi_1$ is parallel.

Case (2-b): $\operatorname{rank}N_1=2$.
In this case, we can assume additionally that $N_1\big|_U=\operatorname{Span}\{\xi_1,\xi_2\}$.
Then we have $\eta_1=\lambda_1\xi_1+\lambda_2\xi_2$ and $\eta_2=\mu_1\xi_1+\mu_2\xi_2$.
It follows from (\ref{eqtrH2}-\ref{Prin2}) that $\lambda_2, \mu_2$ are non-zero constants.
%(If $n=2k$, then $\lambda_1=\mu_1=H$ and $\lambda_2=-\mu_2\neq0$.
%We can choose an appropriate orthonormal basis for each tangent space
%to ensure the continuity of $\lambda_2$ and $\mu_2$.)
Then $\eta_1$ and $\eta_2$ are smooth normal vector fields
and hence Dupin principal normal vector fields with multiplicities $k$ and $n-k$, respectively.
Similarly, $E_{\eta_1}$ and $E_{\eta_2}$ are smooth distributions such that $TM=E_{\eta_1}\oplus E_{\eta_2}$.
For every point of $M$, we can choose a smooth orthonormal local frame $\(E_1,\cdots,E_n\)$ such that $E_{\eta_1}=\operatorname{Span}\{E_1,\cdots,E_k\}$ and $E_{\eta_2}=\operatorname{Span}\{E_{k+1},\cdots,E_n\}$.
Note that (\ref{shapeop}) implies that $$k\eta_1+(n-k)\eta_2=n\mathcal{H},$$
and thus we have
$$\operatorname{II}(E_1,E_1)=\eta_1=\frac{n}{k}\mathcal{H}-\frac{n-k}{k}\eta_2
=\frac{1}{k}\sum_{i=1}^n\operatorname{II}(E_i,E_i)-\frac{n-k}{k}\operatorname{II}(E_n,E_n),$$
i.e., $$(k-1)\operatorname{II}(E_1,E_1)-\sum_{i=2}^{n-1}\operatorname{II}(E_i,E_i)
+(n-k-1)\operatorname{II}(E_n,E_n)=0.$$
It follows from Lemma \ref{N1para} that $N_1$ is parallel, i.e.,
$\nabla_X^\bot\xi\in\Gamma(N_1)$ for any $X\in TM$ and any smooth section $\xi\in\Gamma(N_1)$.
By $\langle\xi_1,\xi_1\rangle=\langle\xi_2,\xi_2\rangle=1$, we obtain %and $\langle\xi_1,\xi_2\rangle=0$
\begin{equation}\label{eqno1}
\<\nabla_{E_i}^\bot\xi_1,\xi_1\>=\<\nabla_{E_i}^\bot\xi_2,\xi_2\>=0\ \mbox{for each}\ 1\leq i\leq n.
\end{equation}
Since $\eta_1,\eta_2$ are Dupin principal normal vector fields, we have
\begin{equation}\label{eqno3}
\nabla_{E_i}^\bot\eta_1=\lambda_1\nabla_{E_i}^\bot\xi_1+\lambda_2\nabla_{E_i}^\bot\xi_2=0\
\mbox{for each}\ 1\leq i\leq k,
\end{equation}
\begin{equation}\label{eqno4}
\nabla_{E_j}^\bot\eta_2=\mu_1\nabla_{E_j}^\bot\xi_1+\mu_2\nabla_{E_j}^\bot\xi_2=0\
\mbox{for each}\ k+1\leq j\leq n.
\end{equation}
Combining (\ref{eqno1}-\ref{eqno4}) and the fact that $N_1$ is parallel, we have
$\nabla_{E_i}^\bot\xi_1=0$ for each $1\leq i\leq n$, and thus $\mathcal{H}=H\xi_1$ is parallel.

To sum up, we have already proved that $M$ is an Einstein submanifold with flat normal bundle and parallel mean curvature vector field.
Then by \cite[Theorem 1]{Onti18}, there exist an isometric immersion
$$g: M^n\rightarrow S^{n_1}\(r_1\)\times\cdots\times S^{n_l}\(r_l\)\subset S^{n+l-1}\big(1/\sqrt{\overline{c}}\,\big)$$
and a totally umbilical embedding
$$u: S^{n+l-1}\big(1/\sqrt{\overline{c}}\,\big)\rightarrow F^{n+m}(c)$$
such that $f=u\circ g$ and $\alpha(n,k,H,c)=(n-l)\overline{c}$ for an positive integer $l$,
where $r_i^2=\frac{n_i-1}{\alpha(n,k,H,c)}$, $2\leq n_i\leq n$ and $\sum_{i=1}^ln_i=n$.
Since $u$ is totally umbilical, it is easy to verify that
$H^2=H_g^2+H_u^2$ and $\overline{c}=c+H_u^2$,
where $H_g$ and $H_u$ denote the mean curvature of $g$ and $u$, respectively.
It follows that $\alpha(n,k,H,c)=(n-l)\(c+H_u^2\)\leq(n-l)\(c+H^2\)$.
Then by Lemma \ref{alpha}, we have $n-3<n-l$, i.e., $l=1$ or $2$.
Recall that $H_k(M;\mathbb{Z})=H_{n-k}(M;\mathbb{Z})=0$ does not hold in Case (2).
Hence, we must have $l=2$ and $M$ is isometric to the Einstein hypersurface
$S^{k}\(\sqrt{\frac{k-1}{n-2}}\,r\)\times S^{n-k}\(\sqrt{\frac{n-k-1}{n-2}}\,r\)$
in the totally umbilical sphere $S^{n+1}\(r\)\subset F^{n+m}(c)$, where $r=\sqrt{\frac{n-2}{\alpha(n,k,H,c)}}$.
Finally, one can verify by a direct calculation that the Ricci curvature and the mean curvature of
this Einstein hypersurface are exactly $$\operatorname{Ric}=\alpha(n,k,H,c)=\frac{n-2}{r^2}$$ and
$$H_g=\sqrt{H^2+c-\overline{c}}=\sqrt{\(\varphi(k)-\frac1{n-2}\)\(c+H^2\)}=
\frac{n-2k}{rn\sqrt{(k-1)(n-k-1)}}.$$
\end{proof}

\subsection{Proof of Corollary \ref{CorA}}
\begin{proof}
Note that Lemma \ref{alpha} implies that $(n-2)\(c+H^2\)\geq\alpha(n,\[\frac{n}2\],H,c)$
and the equality holds if and only if $n=2k$.
Then by Theorem \ref{ThmA} with $k=\[\frac{n}2\]$,
$M$ is either a simply connected homology sphere or isometric to the minimal Clifford hypersurface
$S^{k}\(r/\sqrt{2}\)\times S^{k}\(r/\sqrt{2}\)$
in the totally umbilical sphere $S^{n+1}\(r\)$, where $n=2k$ and $r=\frac{1}{\sqrt{c+H^2}}$.
It follows from the Hurewicz theorem, the Whitehead theorem and the generalized Poincar\'{e} conjecture that every simply connected homology sphere is a topological sphere.
\end{proof}

\subsection{Proof of Corollary \ref{CorB}}
\begin{proof}
Note that for $n=2k+1$, we have $$\alpha(n,k,H,c)=\(n-2-\frac{4}{n^3-2n^2-n-2}\)\(c+H^2\).$$
Then the proof is completed by the same argument as in the proof of Corollary \ref{CorA}.
\end{proof}

%%%%%%%%%%%%%%%%%%%%%%%
\begin{acknow}
The authors would like to thank Professor Marcos Dajczer, Professor Theodoros Vlachos and Professor Rodrigo Montes for their interest and comments.
\end{acknow}

%\begin{data}
%This manuscript has no associated data.
%\end{data}
%\begin{coi}
%The authors declare that there is no conflict of interest.
%\end{coi}

%%%%%%%%%%%%%%%%%%%%%%%%%%%%%%%%%%%%%%%%%%%%

\end{document}